\documentclass[11pt,twoside]{amsart}
\usepackage{latexsym,amssymb,amsmath}
\usepackage[pagebackref]{hyperref}
\numberwithin{equation}{section}
\newtheorem{theorem}{Theorem}[section]
\newtheorem{lemma}[theorem]{Lemma}
\newtheorem{proposition}[theorem]{Proposition}
\newtheorem{corollary}[theorem]{Corollary}
\newtheorem{conjecture}[theorem]{Conjecture}

\theoremstyle{definition}
\newtheorem{definition}[theorem]{Definition}

\newtheorem{example}[theorem]{Example}

\begin{document}


\title[Vanishing ideals]{Vanishing ideals over finite fields} 

\author{Azucena Tochimani}
\address{
Departamento de
Matem\'aticas\\
Centro de Investigaci\'on y de Estudios
Avanzados del
IPN\\
Apartado Postal
14--740 \\
07000 Mexico City, D.F.
}
\email{tochimani@math.cinvestav.mx}
\thanks{The first author was partially supported by CONACyT. The
second author was partially supported by SNI}

\author{Rafael H. Villarreal}
\address{
Departamento de
Matem\'aticas\\
Centro de Investigaci\'on y de Estudios
Avanzados del
IPN\\
Apartado Postal
14--740 \\
07000 Mexico City, D.F.
}
\email{vila@math.cinvestav.mx}



\begin{abstract} 
Let $\mathbb{F}_q$ be a finite field, let $\mathbb{X}$ be a subset
of a projective space ${\mathbb P}^{s-1}$, over the 
field $\mathbb{F}_q$, parameterized by rational functions, and let 
$I(\mathbb{X})$ be the vanishing ideal of $\mathbb{X}$. The main
result of this paper is a formula for $I(\mathbb{X})$ that will
allows us to compute: (i) the algebraic invariants of
$I(\mathbb{X})$, and (ii) the basic 
parameters of the corresponding Reed-Muller-type code.
\end{abstract}

\maketitle

\section{Introduction}\label{section-intro}

In this paper we study vanishing ideals of sets in 
projective spaces parameterized by rational
functions over finite fields. 

Let $R=K[\mathbf{y}]=K[y_1,\ldots,y_n]$ be a polynomial ring over 
a finite field $K=\mathbb{F}_q$ and
let $F$  be a finite set $\{f_1/g_1,\ldots,f_s/g_s\}$ of rational
functions in $K(\mathbf{y})$, the quotient field of $R$, where $f_i$ (resp. $g_i$)
is in $R$ (resp. $R\setminus\{0\}$) for all $i$. As
usual we denote the affine and projective spaces over the field $K$ by
$\mathbb{A}^s$ and $\mathbb{P}^{s-1}$, respectively. Points of the
projective space ${\mathbb P}^{s-1}$ are denoted by $[\alpha]$, where $0\neq \alpha\in
K^s$. The {\it projective set parameterized}
$F$, denoted by $\mathbb{X}$, is the set of all points
$$[(f_1(x)/g_1(x),\ldots,f_s(x)/g_s(x))]$$ 
in $\mathbb{P}^{s-1}$ that 
are well defined, i.e., $x\in K^n$, $f_i(x)\neq 0$  for some $i$, and
$g_i(x)\neq 0$ for all $i$.

Let $S=K[t_1,\ldots,t_s]=\oplus_{d=0}^\infty S_d$ be a polynomial ring 
over the field $K$ with the standard grading. The graded ideal 
$I(\mathbb{X})$ generated by the  
homogeneous polynomials of $S$ that vanish at all points of
$\mathbb{X}$
 is called the {\it vanishing ideal\/} of $\mathbb{X}$. 

There are good reasons to study vanishing ideals over 
finite fields. They are used in 
algebraic geometry \cite{harris} and
algebraic coding theory \cite{GRT}. They are also used in 
polynomial interpolation problems \cite{sauer}.

We come to our main result.

\noindent {\bf Theorem}~\ref{puebla-cinvestav-2-rat-finite}{\it\  
Let $B=K[y_0,y_1,\ldots,y_n,z,t_1,\ldots,t_s]$ be a polynomial ring
over $K=\mathbb{F}_q$. If $\mathbb{X}$ is a projective set  
parameterized by rational functions $f_1/g_1,\ldots,f_s/g_s$ in
$K(\mathbf{y})$, then 
$$
I(\mathbb{X})=(\{g_it_i-f_iz\}_{i=1}^s,\{y_i^q-y_i\}_{i=1}^n,
y_0g_1\cdots g_s-1)\cap S.
$$}
\quad Using the computer algebra system {\it
Macaulay\/}$2$ \cite{mac2}, this result can be used to compute the degree,
regularity, and Hilbert polynomial of $I(\mathbb{X})$ (see
Example~\ref{feb7-15}).  

By the algebraic methods
introduced in \cite{algcodes} (see Section~\ref{prelim-section}), this
result can also be used to 
compute the basic parameters (length, dimension, minimum distance) of
the corresponding projective Reed-Muller-type code over $\mathbb{X}$ (see Example~\ref{feb12-15}).  

For all unexplained
terminology and additional information,  we refer to
\cite{CLO,harris,cocoa-book} (for algebraic geometry and 
computational commutative algebra) and 
\cite{tsfasman} (for coding theory).

\section{Preliminaries}\label{prelim-section}
All results of this
section are well-known. To avoid repetitions, we continue to employ
the notations and 
definitions used in Section~\ref{section-intro}.

If $d\in\mathbb{N}$, let
$S_d$ denote the set of homogeneous polynomials of
degree $d$ in $S$, together with the zero
polynomial. Thus $S_d$ is a $K$-linear space and 
$S=\oplus_{d=0}^\infty S_d$. 

\begin{definition}\rm An ideal $I\subset S$ is {\it graded\/} if $I$
is generated 
by homogeneous polynomials.
\end{definition}

\begin{proposition}{\rm\cite[p.~92]{Mats}}\label{jan27-15}
Let $I\subset S$ be an ideal. The 
following conditions are equivalent\/{\rm:}
\begin{itemize}
\item[$(\mathrm{g}_1)$] $I$ is a graded ideal. 
\item[$(\mathrm{g}_2)$] If $f=\sum_{d=0}^rf_d$ is in $I$, $f_d\in
S_d$ for $d=0,\ldots,r$, then each $f_d$ is in $I$.
\end{itemize}
\end{proposition}

Let $I$ be a graded ideal of $S$ of
dimension $k$. As usual, by the {\it dimension\/} of $I$ 
we mean the Krull dimension of $S/I$.
 The {\it Hilbert
function\/} of $S/I$ is the function $H_I\colon\mathbb{N}\rightarrow\mathbb{N}$ given by
$$
H_I(d)=\dim_K(S_d/I_d),
$$
where $I_d=I\cap S_d$. There is a unique
polynomial $h_I(x)\in\mathbb{Q}[x]$ 
of degree $k-1$ such that $h_I(d)=H_I(d)$ for 
$d\gg 0$ \cite[p.~330]{singular-book}. By convention, the zero polynomial 
has degree $-1$. 

The {\it degree\/} or {\it multiplicity\/} of $S/I$ is the 
positive integer 
$$
\deg(S/I):=\left\{\begin{array}{ll}(k-1)!\,\displaystyle \lim_{d\rightarrow\infty}{H_I(d)}/{d^{k-1}}
&\mbox{if }k\geq 1,\\
\dim_K(S/I) &\mbox{if\ }k=0.
\end{array}\right.
$$ 

\begin{definition}\label{definition:index-of-regularity}
The {\it regularity of the Hilbert function\/} of $S/I$, or simply the
\emph{regularity} of $S/I$, denoted 
${\rm reg}(S/I)$, is the least integer $r\geq 0$ such that
$H_I(d)$ is equal to $h_I(d)$ for $d\geq r$.  
\end{definition}

We will use the following multi-index notation: for
$a=(a_1,\ldots,a_s)\in\mathbb{Z}^s$, set $t^a=t_1^{a_1}\cdots
t_s^{a_s}$. We call $t^a$ a {\it Laurent monomial\/}. 
If $a_i\geq 0$ for all $i$,
$t^a$ is a {\it monomial\/} of $S$. An ideal of $S$ generated 
by polynomials of the form $t^a-t^b$, with $a,b$ in $\mathbb{N}^s$, 
is called a {\it
binomial ideal\/} of $S$.

\begin{lemma}{\cite[p.~321]{monalg-rev}}\label{may1-1-11}
Let $B=K[y_1,\ldots,y_n,t_1,\ldots ,t_s]$ be a polynomial ring over a
field $K$. If  $I$ is a binomial ideal of $B$, then the reduced Gr\"obner basis of
 $I$ with respect to any term order consists of binomials and 
$I\cap K[t_1,\ldots,t_s]$ is a binomial ideal of $S$.
\end{lemma}

\begin{proposition}{\rm\cite[pp.~136--137]{JacI}}\label{oct14-14} 
Let $K=\mathbb{F}_q$ be a finite field and let $\mathbb{A}^s$ be the
affine space of dimension $s$ over $K$. Then $I(\mathbb{A}^s)=
(\{t_i^q-t_i\}_{i=1}^s)$.
\end{proposition}

\paragraph{\bf Projective Reed-Muller-type codes} In this part we introduce
the family of projective Reed-Muller-type codes and its connection to vanishing
ideals and Hilbert functions. 

Let $K=\mathbb{F}_q$ be a finite field and let $\mathbb{Y}=\{P_1,\ldots,P_m\}\neq\emptyset$ be a
subset of $\mathbb{P}^{s-1}$ with $m=|\mathbb{Y}|$. 
Fix a degree $d\geq 1$. For each $i$ there is $f_i\in S_d$ such that
$f_i(P_i)\neq 0$. There is a well-defined
$K$-linear map: 
\begin{equation}\label{ev-map}
{\rm ev}_d\colon S_d=K[t_1,\ldots,t_s]_d\rightarrow K^{|\mathbb{Y}|},\ \ \ \ \ 
f\mapsto
\left(\frac{f(P_1)}{f_1(P_1)},\ldots,\frac{f(P_m)}{f_m(P_m)}\right).
\end{equation}

The image of 
$S_d$ under ${\rm ev}_d$, denoted by  $C_\mathbb{Y}(d)$, is called a {\it
projective Reed-Muller-type code\/} of degree $d$ over the set
$\mathbb{Y}$ \cite{GRT}.  There is an isomorphism of $K$-vector spaces
$S_d/I(\mathbb{Y})_d\simeq C_\mathbb{Y}(d)$. It is usual to
denote the Hilbert function $S/I(\mathbb{Y})$ by $H_\mathbb{Y}$. Thus
$H_\mathbb{Y}(d)$ is equal to $\dim_KC_\mathbb{Y}(d)$.  
The {\it minimum distance\/} of the linear code $C_\mathbb{Y}(d)$, denoted
$\delta_\mathbb{Y}(d)$, is given by 
$$\delta_\mathbb{Y}(d):=\min\{\omega(v)
\colon 0\neq v\in C_\mathbb{Y}(d)\},
$$
where $\omega(v)$ is the {\it Hamming
weight\/} of $v$, 
that is, $\omega(v)$ is the number of non-zero 
entries of $v$. 
\begin{definition}\label{basic-parameters-def} The {\it basic parameters} of the linear
code $C_\mathbb{Y}(d)$ are: its {\it length\/} $|\mathbb{Y}|$, {\it
dimension\/} 
$\dim_K C_\mathbb{Y}(d)$, and {\it minimum distance\/} $\delta_\mathbb{Y}(d)$.
\end{definition}

The following summarizes the well-known relation between 
projective Reed-Muller-type codes and the theory of Hilbert functions.

\begin{proposition}{\rm(\cite{GRT}, \cite{algcodes})}\label{jan4-15}
The following hold. 
\begin{itemize}
\item[{\rm (i)}] $H_\mathbb{Y}(d)=\dim_KC_\mathbb{Y}(d)$ for $d\geq 0$.
\item[{\rm(ii)}] ${\rm deg}(S/I(\mathbb{Y}))=|\mathbb{Y}|$. 
\item[{\rm (iii)}] $\delta_\mathbb{Y}(d)=1$ for $d\geq {\rm reg}(S/I(\mathbb{Y}))$.  
\end{itemize}
\end{proposition} 

\section{Rational parameterizations over finite fields}

We continue to employ the notations and 
definitions used in Sections~\ref{section-intro} and
\ref{prelim-section}. 
Throughout this section $K=\mathbb{F}_q$ is a finite field and 
$\mathbb{X}$ is the projective set parameterized
by the rational functions $F=\{f_1/g_1,\ldots,f_s/g_s\}$ in $K(\mathbf{y})$. 

\begin{theorem}{\rm(Combinatorial Nullstellensatz
\cite{alon-cn})}\label{comb-null} Let $S=K[t_1,\ldots,t_s]$ be a
polynomial ring over a field $K$, let $f\in S$, and let
$a=(a_i)\in\mathbb{N}^s$. Suppose that the coefficient of
$t^a$ in $f$ is non-zero and $\deg\left(f\right)=a_1+\cdots+a_s$. If
$A_{1},\ldots ,A_{s}$ are subsets of $K$, with $\left|A_{i}\right| > a_i$ for
all $i$, then there are $x_{1}\in A_{1},\ldots,x_s\in A_s$ such that
$f\left(x_{1},\ldots ,x_{s}\right) \neq 0$.  
\end{theorem}

\begin{lemma}\label{dec12-11} Let $K$ be a field and let $A_1,\ldots,A_s$ be 
a collection of non-empty finite subsets of $K$. 
If $Y:=A_1\times\cdots\times A_s\subset\mathbb{A}^{s}$, $g\in I(Y)$ and $\deg_{t_{i}}\left(g\right)
<|A_i|$ for $i=1,\ldots ,s$, then $g=0$. In particular if $g$ is a
polynomial of $S$ that vanishes at all points of $\mathbb{A}^s$, then
$g=0$.
\end{lemma}

\begin{proof} We 
proceed by contradiction. Assume that $g$ is not zero. Then, there is
a monomial 
$t^a=t_1^{a_1}\cdots t_s^{a_s}$ of $g$ with 
$\deg(g)=a_1+\cdots+a_s$ and $a=(a_1,\ldots,a_s)\neq 0$. 
As $\deg_{t_{i}}(g)<|A_i|$ for all $i$, then $a_i<\left|A_{i}\right|$ for all $i$. Thus, by 
Theorem~\ref{comb-null}, there are $x_{1},\ldots,x_{s}$ with $x_i\in
A_i$ for all $i$ such that 
$g\left( x_{1},\ldots ,x_{s}\right) \neq 0$, a contradiction to the
assumption that $g$ vanishes on $Y$. 
\end{proof}

\begin{definition} The {\it affine set parameterized\/} by $F$,
denoted $\mathbb{X}^*$, is the set of all points 
$$(f_1(x)/g_1(x),\ldots,f_s(x)/g_s(x))$$ 
in $\mathbb{A}^{s}$ such that
$x\in K^n$ and $g_i(x)\neq 0$ for all $i$. 
\end{definition}

\begin{lemma}\label{empty-affine} Let $K=\mathbb{F}_q$ be a finite field. 
The following conditions are equivalent{\rm:}
\begin{itemize}
\item[\rm(a)] $g_1\cdots g_s$ vanishes at all points of $K^n$.
\item[\rm(b)] $g_1\cdots g_s\in(\{y_i^q-y_i\}_{i=1}^n)$.
\item[\rm(c)] $(\{g_it_i-f_iz\}_{i=1}^s,\{y_i^q-y_i\}_{i=1}^n,
y_0g_1\cdots g_s-1)\cap S=S$.
\item[\rm(d)] $\mathbb{X}^*=\emptyset$.
\end{itemize}
\end{lemma}

\begin{proof} (a) $\Leftrightarrow$ (b)): This follows at once from
Proposition~\ref{oct14-14}. 

(a) $\Leftrightarrow$ (d)): This follows from the definition of
$\mathbb{X}^*$.

(c) $\Rightarrow$ (a)): We can write 
$1=\sum_{i=1}^sa_i(g_it_i-f_iz)+\sum_{j=1}^nb_j(y_j^q-y_j)+h(y_0g_1\cdots
g_s-1)$, where the $a_i$'s, $b_j$'s and $h$ are polynomials in the
variables $y_j$'s, $t_i$'s, $y_0$ and $z$. Take an arbitrary point
$x=(x_i)$ in $K^n$. In the equality above, 
making $y_i=x_i$ for all $i$, $z=0$ and $t_i=0$
for all $i$, we get that $1=h_1(y_0g_1(x)\cdots g_s(x)-1)$ for some
$h_1$. If $(g_1\cdots g_s)(x)\neq 0$, then $h_1(y_0g_1(x)\cdots g_s(x)-1)$ is a
polynomial in $y_0$ of positive degree, a contradiction. Thus
$(g_1\cdots g_s)(x)=0$.

(b) $\Rightarrow$ (c)): Writing $g_1\cdots
g_s=\sum_{j=1}^nb_j(y_j^q-y_j)$, we get 
$y_0g_1\cdots g_s-1=-1+\sum_{j=1}^ny_0b_j(y_j^q-y_j)$. Thus $1$ is in
the ideal $(\{g_it_i-f_iz\}_{i=1}^s,\{y_i^q-y_i\}_{i=1}^n,
y_0g_1\cdots g_s-1)$.
\end{proof}

\begin{lemma}\label{may16-14} Let $f_1/g_1,\ldots,f_s/g_s$ be rational
functions of $K(\mathbf{y})$ and let $f=f(t_1,\ldots,t_s)$ be a
polynomial in $S$ of degree $d$. Then 
$$
g_1^{d+1}\cdots g_s^{d+1}f=\sum_{i=1}^sg_1\cdots
g_sh_i(g_it_i-f_i)+g_1^{d+1}\cdots g_s^{d+1}f(f_1/g_1,\ldots,f_s/g_s) 
$$
for some $h_1,\ldots,h_s$ in the polynomial ring 
$K[y_1,\ldots,y_n,t_1,\ldots,t_s]$. If $f$ is homogeneous and $z$ is
a new variable, then 
$$
g_1^{d+1}\cdots g_s^{d+1}f=\sum_{i=1}^sg_1\cdots
g_sh_i(g_it_i-f_iz)+g_1^{d+1}\cdots g_s^{d+1}z^df(f_1/g_1,\ldots,f_s/g_s) 
$$
for some $h_1,\ldots,h_s$ in the polynomial ring 
$K[y_1,\ldots,y_n,z,t_1,\ldots,t_s]$.
\end{lemma}
\begin{proof} We can write $f=\lambda_1 t^{m_1}+\cdots+\lambda_r t^{m_r}$ with 
$\lambda_i\in K^*$ and $m_i\in \mathbb{N}^s$ for all $i$.  Write
$m_i=(m_{i1},\ldots,m_{is})$ for $1\leq i\leq r$ and set
$I=(\{g_it_i-f_i\}_{i=1}^s)$. By the binomial theorem, for all $i,j$,
we can write 
$$
t_j^{m_{ij}}=\left[(t_j-(f_j/g_j))+(f_j/g_j)\right]^{m_{ij}}=(h_{ij}/g_j^{m_{ij}})+(f_j/g_j)^{m_{ij}},
$$
for some $h_{ij}\in I$. Hence for any $i$ we can write 
$$t^{m_i}=t_1^{m_{i_1}}\cdots t_s^{m_{i_s}}=
(G_i/g_1^{m_{i1}}\cdots g_s^{m_{is}})+(f_1/g_1)^{m_{i1}}\cdots
(f_s/g_s)^{m_{is}},$$ 
where $G_i\in I$. Notice that
$m_{i_1}+\cdots+m_{is}\leq d$ for all $i$ because $f$ has degree $d$. 
Then substituting these expressions for
$t^{m_{1}},\ldots,t^{m_{s}}$ in 
$f=\lambda_1 t^{m_1}+\cdots+\lambda_r t^{m_r}$ and multiplying $f$ by
$g_1^{d+1}\cdots g_s^{d+1}$, we obtain the required expression. 

If $f$ is homogeneous of degree $d$, the required
expression for $g_1^{d+1}\cdots g_s^{d+1}f$ follows from the first part by considering the
rational functions $f_1z/g_1,\ldots,f_sz/g_s$, i.e., by replacing
$f_i$ by $f_iz$, and observing that $f(f_1z,\ldots,f_sz)=z^df(f_1,\ldots,f_s)$. 
\end{proof}

\begin{lemma}\label{zero-affine} Let $K=\mathbb{F}_q$ be a finite field. 
The following conditions are equivalent{\rm:}
\begin{itemize}
\item[\rm(a)] $(\{g_it_i-f_iz\}_{i=1}^s,\{y_i^q-y_i\}_{i=1}^n,
y_0g_1\cdots g_s-1)\cap S=(t_1,\ldots,t_s)$.
\item[\rm(b)] $\mathbb{X}^*=\{0\}$.
\end{itemize}
\end{lemma}

\begin{proof}
(a) $\Rightarrow$ (b)): By Lemma~\ref{empty-affine}, 
$\mathbb{X}^*\neq\emptyset$. Take a point $P$ in $\mathbb{X}^*$, 
i.e., there is $x=(x_i)\in \mathbb{A}^s$ such that $g_i(x)\neq 0$ for
all $i$ and $P=(f_1(x)/g_1(x),\ldots,f_s(x)/g_s(x))$. By hypothesis,
for each $t_k$, we
can write 
\begin{equation}\label{feb18-15}
t_k=\sum_{i=1}^sa_i(g_it_i-f_iz)+\sum_{j=1}^nb_j(y_j^q-y_j)+h(y_0g_1\cdots
g_s-1),
\end{equation}
where the $a_i$'s, $b_j$'s and $h$ are polynomials in the
variables $y_j$'s, $t_i$'s, $y_0$ and $z$. From Eq.~(\ref{feb18-15}),
making $y_i=x_i$ for all
$i$, $y_0=1/g_1(x)\cdots g_s(x)$, $t_i=f_i(x)/g_i(x)$ for all $i$, and
$z=1$, we get that $f_k(x)/g_k(x)=0$. Thus $P=0$. 

(b) $\Rightarrow$ (a)): Setting $I=(\{g_it_i-f_iz\}_{i=1}^s,\{y_i^q-y_i\}_{i=1}^n,
y_0g_1\cdots g_s-1)$, by Lemma~\ref{empty-affine} one has that $I\cap
S\subsetneq S$. Thus it suffices to show that $t_k\in I\cap S$ for all
$k$. Notice that $g_1\cdots g_sf_k$ vanishes at all points of
$\mathbb{A}^s$ because $\mathbb{X}^*=\{0\}$. Hence, thanks to
Proposition~\ref{oct14-14}, $g_1\cdots g_sf_k$
is in $(\{y_i^q-y_i\}_{i=1}^n)$. Setting $w=y_0g_1\cdots
g_s-1$, and applying Lemma~\ref{may16-14} with $f=t_k$, we can write
\begin{equation*}
(w+1)^{2}t_k=\sum_{i=1}^sy_0^{2}g_1\cdots g_sh_i(g_it_i-f_iz)+
y_0^{2}g_1^{2}\cdots
g_s^{2}z(f_k/g_k).
\end{equation*}
Therefore $(w+1)^2t_k\in I$. Thus $t_k\in I\cap S$.  
\end{proof}

\begin{lemma}\label{jan27-15-1} If $I=(\{g_it_i-f_iz\}_{i=1}^s,\{y_i^q-y_i\}_{i=1}^n,
y_0g_1\cdots g_s-1)$ and $\mathfrak{m}=(t_1,\ldots,t_s)$ is the
irrelevant maximal ideal of $S$, then 
\begin{itemize}
\item[\rm(a)] $I\cap S$ is a graded ideal of $S$, and 
\item[\rm(b)] $\mathbb{X}\neq\emptyset$ if and only if
$I\cap S\subsetneq\mathfrak{m}$.
\end{itemize}
\end{lemma}

\begin{proof} (a): We set $B=K[y_0,y_1,\ldots,y_n,z,t_1,\ldots,t_s]$.
Take $0\neq f\in I\cap S$ and write
it as $f=f_1+\cdots+f_r$, where $f_i$ is a homogeneous polynomial of
degree $d_i$ and $d_1<\cdots<d_r$. By induction, using
Proposition~\ref{jan27-15}, it suffices to show
that $f_r\in I\cap S$. We can write 
$$
f=\sum_{i=1}^sa_i(g_it_i-f_iz)+\sum_{i=1}^nc_i(y_i^q-y_i)+c(y_0g_1\cdots
g_s-1),
$$
where the $a_i$'s, $c_i$'s, and $c$ are in $B$. 
Making the substitution $t_i\rightarrow t_iv$, $z\rightarrow zv$, with
$v$ an extra variable, and regarding
$f(t_1v,\ldots,t_sv)$ as a polynomial in $v$ it follows readily that
$v^{d_r}f_r$ is in the ideal generated by
$\mathcal{B}=\{g_it_iv-f_izv\}_{i=1}^s\cup\{y_i^q-y_i\}_{i=1}^n
\cup\{y_0g_1\cdots g_s-1\}$. Writing $v^{d_r}f_r$  as a linear
combination of $\mathcal{B}$, with coefficients in $B$, and making
$v=1$, 
we get that 
$f_r\in I\cap S$.

(b): $\Rightarrow$) If $\mathbb{X}\neq\emptyset$, by
Lemma~\ref{empty-affine}, we get that $I\cap S\neq S$. By part (a) the
ideal $I\cap S$ is graded. Hence $I\cap S\subsetneq\mathfrak{m}$. 

$\Leftarrow$) If $I\cap S\subsetneq\mathfrak{m}$, by 
Lemmas~\ref{empty-affine} and \ref{zero-affine}, we get 
$\mathbb{X}^*\neq\emptyset$ and $\mathbb{X}^*\neq\{0\}$. Thus
$\mathbb{X}\neq\emptyset$.   
\end{proof}

We come to the main result of this paper.

\begin{theorem}\label{puebla-cinvestav-2-rat-finite}
Let $B=K[y_0,y_1,\ldots,y_n,z,t_1,\ldots,t_s]$ be a polynomial ring
over a finite field $K=\mathbb{F}_q$. If $\mathbb{X}$ is a projective set  
parameterized by rational functions $f_1/g_1,\ldots,f_s/g_s$ in
$K(\mathbf{y})$ and $\mathbb{X}\neq\emptyset$,  
then 
$$
I(\mathbb{X})=(\{g_it_i-f_iz\}_{i=1}^s,\{y_i^q-y_i\}_{i=1}^n,
y_0g_1\cdots g_s-1)\cap S.
$$
\end{theorem}

\begin{proof} We set
$I=(\{g_it_i-f_iz\}_{i=1}^s,\{y_i^q-y_i\}_{i=1}^n, 
y_0g_1\cdots g_s-1)$. First we show the
inclusion $I(\mathbb{X})\subset I\cap S$.  
Take a homogeneous polynomial $f=f(t_1,\ldots,t_s)$ of degree $d$
that vanishes at all points of $\mathbb{X}$. Setting $w=y_0g_1\cdots
g_s-1$, by Lemma~\ref{may16-14}, we can write
\begin{equation}\label{may20-14-finite}
(w+1)^{d+1}f=\sum_{i=1}^sy_0^{d+1}g_1\cdots
g_sa_i(g_it_i-f_iz)+z^dy_0^{d+1}g_1^{d+1}\cdots 
g_s^{d+1}f(f_1/g_1,\ldots,f_s/g_s),
\end{equation}
where $a_1,\ldots,a_s$ are in $B$. We set 
$H=g_1^{d+1}\cdots g_s^{d+1}f(f_1/g_1,\ldots,f_s/g_s)$. This is a
polynomial in $K[\mathbf{y}]$. Thus, by the division algorithm
in $K[\mathbf{y}]$ (see \cite[Theorem~3, p.~63]{CLO}), we can
write 
\begin{equation}\label{23-jul-10-1}
H=H(y_1,\ldots,y_n)=\sum_{i=1}^nh_i(y_i^{q}-y_i)+G(y_1,\ldots,y_n)
\end{equation}
for some $h_1,\ldots,h_n$ in $K[\mathbf{y}]$, 
where the monomials that occur in $G=G(y_1,\ldots,y_n)$ are not divisible by 
any of the monomials $y_1^{q},\ldots,y_n^{q}$, i.e.,
$\deg_{y_i}(G)<q$ for $i=1,\ldots,n$.
Therefore, using Eqs.~(\ref{may20-14-finite}) and (\ref{23-jul-10-1}), we
obtain the equality
\begin{equation}\label{23-jul-10-2}
(w+1)^{d+1}f=\sum_{i=1}^sy_0^{d+1}g_1\cdots
g_sa_i(g_it_i-f_iz)+z^dy_0^{d+1}\sum_{i=1}^nh_i(y_i^{q}-y_i)+
z^dy_0^{d+1}G(y_1,\ldots,y_n).
\end{equation}
Thus to show that $f\in I\cap S$ we need only show that $G=0$. We
claim that $G$ vanishes on $K^n$. Notice that $y_i^q-y_i$ vanishes at
all points of $K^n$ because $(K^*,\,\cdot\, )$ is a group of order $q-1$. 
Take an arbitrary sequence $x_1,\ldots,x_n$
of elements of $K$, i.e., $x=(x_i)\in K^n$. 

Case (I): $g_i(x)=0$ for some $i$. Making $y_j=x_j$ for all $j$ in
Eq.~(\ref{23-jul-10-2}) we get $G(x)=0$.

Case (II): $f_i(x)=0$ and $g_i(x)\neq 0$ for
all $i$. Making $y_k=x_k$ and 
$t_j=f_j(x)/g_j(x)$ for all $k,j$ in Eq.~(\ref{23-jul-10-2}) and
using 
that $f$ is homogeneous, we obtain that $G(x)=0$.

Case (III): $f_i(x)\neq 0$ for some $i$ and
$g_\ell(x)\neq 0$ for all $\ell$. Making $y_k=x_k$,
$t_j=f_j(x)/g_j(x)$ and $z=1$ in  Eq.~(\ref{23-jul-10-2}) and
using that $f$ vanishes on $[(f_1(x)/g_1(x),\ldots,f_s(x)/g_s(x))]$,
we get that $G(x)=0$. 
This completes the proof of the 
claim. 

Therefore $G$ vanishes at all points of $K^n$ and $\deg_{y_i}(G)<q$ 
for all $i$. Hence, by Lemma~\ref{dec12-11}, we get that $G=0$.  

Next we show the inclusion $I(\mathbb{X})\supset I\cap S$. By 
Lemma~\ref{jan27-15-1} the ideal $I\cap S$ is graded. 
Let $f$ be a homogeneous polynomial of
$I\cap S$. Take a point $[P]$ in $\mathbb{X}$ with
$P=(f_1(x)/g_1(x),\ldots,f_s(x)/g_s(x))$. Writing $f$ as a linear
combination of $\{g_it_i-f_iz\}_{i=1}^s,\{y_i^q-y_i\}_{i=1}^n, 
y_0g_1\cdots g_s-1)$, with coefficients in $K$, and making $t_i=f_i(x)/g_i(x)$, $y_j=x_j$, $z=1$ and
$y_0=1/g_1(x)\cdots g_s(x)$ for all $i,j$ it follows that $f(P)=0$.
Thus $f$ vanishes on $\mathbb{X}$. 
 \end{proof}

\begin{definition}\label{colon-def}\rm 
If $I\subset S$ is an ideal and $h\in S$, we set $(I\colon h):=\{f\in
S\vert\, fh\in I\}$. This ideal is called the {\it colon ideal} of $I$
with respect to $h$. 
\end{definition}

\begin{definition} The {\it projective algebraic set parameterized\/}
by $F$, denoted by $X$, is the set of all points 
$[(f_1(x)/g_1(x),\ldots,f_s(x)/g_s(x))]$ in $\mathbb{P}^{s-1}$ such 
that $x\in K^n$ and $f_i(x)g_i(x)\neq 0$ for all $i$. 
\end{definition}

The ideal $I(X)$ can be computed from $I(\mathbb{X})$ using the colon
operation.

\begin{proposition}\label{jul30-14} If $X\neq\emptyset$, then $(I(\mathbb{X})\colon
t_1\cdots t_s)=I(X)$.
\end{proposition}

\begin{proof} 
Since $X\subset\mathbb{X}$, one has
$I(\mathbb{X})\subset I(X)$. Consequently 
$(I(\mathbb{X})\colon t_1\cdots t_s)\subset I(X)$ because $t_i$ is not
a zero-divisor of $S/I(X)$ for all $i$.  
To show the
reverse inclusion take a homogeneous
polynomial $f$ in $I(X)$. Let
$[P]$ be a point in $\mathbb{X}$, with $P=(\alpha_1,\ldots,\alpha_s)$
and $\alpha_k\neq 0$ for some $k$, and let
$I_{[P]}$ be the 
ideal generated by the homogeneous polynomials of $S$ that vanish 
at $[P]$. Then $I_{[P]}$ is a prime ideal of height $s-1$, 
\begin{equation}\label{primdec-ix}
I_{[P]}=(\{\alpha_kt_i-\alpha_it_k\vert\, k\neq i\in\{1,\ldots,s\}),\
I(\mathbb{X})=\bigcap_{[Q]\in \mathbb{X}}I_{[Q]},
\end{equation}
and the latter is the primary decomposition of $I(\mathbb{X})$.
Noticing that $t_i\in I_{[P]}$ if and only if $\alpha_i=0$, 
it follows that $t_1\cdots t_sf\in I(\mathbb{X})$. Indeed if
$[P]$ has at least one entry equal to zero, then $t_1\cdots t_s\in
I_{[P]}$ and if all entries of $P$ are not zero, then $f\in
I(X)\subset I_{[P]}$.
In either case $t_1\cdots t_sf\in I(\mathbb{X})$. Hence
$f\in(I(\mathbb{X})\colon t_1\cdots t_s)$. 
\end{proof}

Next we present some other means to compute the vanishing ideal $I(X)$.

\begin{theorem}\label{puebla-cinvestav-2-rat-finite-rest}
Let $B=K[y_0,w,y_1,\ldots,y_n,z,t_1,\ldots,t_s]$ be a polynomial ring
over $K=\mathbb{F}_q$. If $X$ is a projective
algebraic set parameterized by $f_1/g_1,\ldots,f_s/g_s$ in
$K(\mathbf{y})$ and $X\neq\emptyset$, then 
\begin{eqnarray*}
I(X)&=&(\{g_it_i-f_iz\}_{i=1}^s,\{y_i^q-y_i\}_{i=1}^n,
y_0g_1\cdots g_s-1,wf_1\cdots f_s-1)\cap S\\
&=&(\{g_it_i-f_iz\}_{i=1}^s,\{y_i^q-y_i\}_{i=1}^n,
\{f_i^{q-1}-1\}_{i=1}^s,y_0g_1\cdots g_s-1)\cap S.
\end{eqnarray*}
\end{theorem}

\begin{proof} This follows adapting the proof of
Theorem~\ref{puebla-cinvestav-2-rat-finite}.
\end{proof}

\begin{theorem}\label{puebla-cinvestav-2-rat-finite-no-rest}
Let $B=K[y_0,y_1,\ldots,y_n,t_1,\ldots,t_s]$ be a polynomial ring
over $K=\mathbb{F}_q$. If $\mathbb{X}^*$ is an affine set  
parameterized by $f_1/g_1,\ldots,f_s/g_s$ in
$K(\mathbf{y})$, then 
$$
I(\mathbb{X}^*)=(\{g_it_i-f_i\}_{i=1}^s,\{y_i^q-y_i\}_{i=1}^n,
y_0g_1\cdots g_s-1)\cap S.
$$
\end{theorem}

\begin{proof} This follows adapting the proof of
Theorem~\ref{puebla-cinvestav-2-rat-finite}.
\end{proof}

\begin{definition} The {\it affine algebraic set parameterized\/} by
$F$, denoted $X^*$, is the set of all points
$(f_1(x)/g_1(x),\ldots,f_s(x)/g_s(x))$ in $\mathbb{A}^{s}$ such that $x\in
K^n$ and $f_i(x)g_i(x)\neq 0$ for all $i$. .
\end{definition}

The ideal $I(X^*)$ can be computed from $I(\mathbb{X}^*)$ using the colon
operation.  

\begin{proposition}\label{jul30-14-1} $(I(\mathbb{X}^*)\colon t_1\cdots
t_s)=I(X^*)$.
\end{proposition}

\begin{proof} This follows adapting the proof of
Proposition~\ref{jul30-14}. 
\end{proof}

\begin{corollary}\label{puebla-cinvestav-finite} 
Let $B=K[t_1,\ldots,t_s,y_1,\ldots,y_n,z]$ be 
a polynomial ring over the finite field $K=\mathbb{F}_q$ and let
$f_1,\ldots,f_s$ be polynomials of $R$. The following hold{\rm :}
\begin{enumerate}
\item[\rm(a)]
If $\mathbb{X}\neq\emptyset$, 
then $I(\mathbb{X})=(\{t_i-f_iz\}_{i=1}^s\cup\{y_i^{q}-y_i\}_{i=1}^n)\cap
S$. 
\item[\rm(b)] 
If $X\neq\emptyset$, then $I(X)=
(\{t_i-f_iz\}_{i=1}^s\cup\{y_i^{q}-y_i\}_{i=1}^n\cup\{f_i^{q-1}-1\}_{i=1}^s
)\cap S$. 
\end{enumerate}
\end{corollary}

\begin{proof} The result follows by adapting the proof of
Theorem~\ref{puebla-cinvestav-2-rat-finite}, and using 
Theorem~\ref{puebla-cinvestav-2-rat-finite-rest}.   
\end{proof}

The formula for $I(X)$ given in (b) can be slightly simplified if the
$f_i$'s are
Laurent monomials  (see \cite[Theorems~2.1 and 2.13]{algcodes}).

\begin{example}\label{feb7-15} Let $f_1=y_1+1$, $f_2=y_2+1$, $f_3=y_1y_2$ and let
$K=\mathbb{F}_5$ be a field with $5$ elements. Using
Corollary~\ref{puebla-cinvestav-finite}, and {\em Macaulay\/}$2$
\cite{mac2}, we get
$$
\begin{array}{llll}
{\rm deg}\, S/I(\mathbb{X})=19,&{\rm deg}\,
S/I(X)=6,\\
\, {\rm reg}\, S/I(\mathbb{X})=5,& \, {\rm reg}\, S/I(X)=2.
\end{array}
$$
For convenience we present the following procedure for {\em Macaulay\/}$2$
\cite{mac2} that we used to compute the degree and the regularity:
\begin{verbatim}
R=GF(5)[z,y1,y2,t1,t2,t3,MonomialOrder=>Eliminate 3];
f1=y1+1,f2=y2+1,f3=y1*y2,q=5
I=ideal(t1-f1*z,t2-f2*z,t3-f3*z,y1^q-y1,y2^q-y2)
Jxx=ideal selectInSubring(1,gens gb I)
I=ideal(t1-f1*z,t2-f2*z,t3-f3*z,y1^q-y1,y2^q-y2,
f1^(q-1)-1,f2^(q-1)-1,f3^(q-1)-1)
Jx=ideal selectInSubring(1,gens gb I)
S=ZZ/5[t1,t2,t3]
Ixx=sub(Jxx,S),Mxx=coker gens Ixx
degree Ixx, regularity Mxx
Ix=sub(Jx,S),Mx=coker gens Ix
degree Ix, regularity Mx
\end{verbatim}
\end{example}

\begin{example}\label{feb12-15} Let $f_1=y_1+1$, $f_2=y_2+1$, $f_3=y_1y_2$ and let
$K=\mathbb{F}_5$ be a field with $5$ elements. Using
Proposition~\ref{jan4-15}, Corollary~\ref{puebla-cinvestav-finite} and {\em Macaulay\/}$2$
\cite{mac2}, we get
\begin{eqnarray*}
\left.
\begin{array}{c|c|c|c|c|c}
 d & 1 & 2 & 3 & 4 &5\\
   \hline
 |\mathbb{X}| & 19 & 19& 19 & 19 &19\\ 
   \hline
 \dim C_{\mathbb{X}}(d)    \   & 3 & 6&10 &15&19\\
   \hline
 \delta_{\mathbb{X}}(d)    \   & 13 & 8& & &1
\end{array}
\right.
& &\left.
\begin{array}{c|c|c}
 d & 1 & 2 \\
   \hline
 |X| & 6 & 6\\ 
   \hline
 \dim C_{X}(d)    \   & 3 & 6\\
  \hline
 \delta_{X}(d)    \   & 3 & 1
\end{array}
\right.
\end{eqnarray*}

The $d$th column of these tables represent the length, the
dimension, and the minimum distance of the projective
Reed-Muller-type codes 
$C_\mathbb{X}(d)$
and $C_X(d)$, respectively (see Section~\ref{prelim-section}). The
minimum distance was computed using the methods of
\cite{hilbert-min-dis}. 
Continuing with the {\em Macaulay\/}$2$ procedure of Example~\ref{feb7-15} we can compute
the other values of these two tables as follows:  
\begin{verbatim}
degree Ixx, regularity Mxx
hilbertFunction(1,Ixx),hilbertFunction(2,Ixx),hilbertFunction(3,Ixx),
hilbertFunction(4,Ixx),hilbertFunction(5,Ixx)
degree Ix, regularity Mx
hilbertFunction(1,Ix),hilbertFunction(2,Ix)
\end{verbatim}
\end{example}

Let us give some application to vanishing ideals over monomial
parameterizations.

\begin{corollary}\label{puebla-cinvestav-finite-monomial} 
Let $K=\mathbb{F}_q$ be a finite field. If $\mathbb{X}$ is
parameterized by Laurent monomials, then $I(\mathbb{X})$ is a radical
Cohen-Macaulay binomial ideal of dimension $1$.
\end{corollary}

\begin{proof} That $I(\mathbb{X})$ is a binomial ideal follows from 
Lemma~\ref{may1-1-11} and applying
Theorem~\ref{puebla-cinvestav-2-rat-finite}. That
$I(\mathbb{X})$ is a radical ideal of dimension $1$ is well known and
follows from Eq.~(\ref{primdec-ix}) (see the proof of
Proposition~\ref{jul30-14}). Recall that ${\rm depth}\, S/I(\mathbb{X})\leq \dim
S/I(\mathbb{X})=1$. From Eq.~(\ref{primdec-ix}) one has that
$\mathfrak{m}=(t_1,\ldots,t_s)$ is not an associated prime of
$I(\mathbb{X})$. Thus ${\rm depth}\, S/I(\mathbb{X})>0$ 
and ${\rm depth}\, S/I(\mathbb{X})=\dim
S/I(\mathbb{X})=1$, i.e., $I(\mathbb{X})$ is Cohen-Macaulay.
\end{proof}

\begin{corollary}{\cite[Theorem~2.1]{algcodes}}\label{ipn-ufpe-cinvestav-1} Let 
$K=\mathbb{F}_q$ be a finite field and let $X$ be a projective
algebraic set parameterized by Laurent monomials. Then 
$I(X)$ is a Cohen-Macaulay lattice ideal and $\dim S/I(X)=1$.
\end{corollary}

\begin{proof} It follows from Proposition~\ref{jul30-14}, 
Theorem~\ref{puebla-cinvestav-2-rat-finite-rest} and
Lemma~\ref{may1-1-11}. 
\end{proof}

\paragraph{\bf Binomial vanishing ideals} 
Let $K$ be a field. The projective space
$\mathbb{P}^{s-1}\cup\{[0]\}$) together with the zero vector $[0]$ is
a monoid under componentwise
multiplication, where $[\mathbf{1}]=[(1,\ldots,1)]$ is the identity
of $\mathbb{P}^{s-1}\cup\{[0]\}$. Recall that monoids always have an
identity element.

\begin{lemma}\label{feb3-15} Let $K=\mathbb{F}_q$ be a finite field and let
$\mathbb{Y}$ be a subset of $\mathbb{P}^{s-1}$. If
$\mathbb{Y}\cup\{[0]\}$  is a submonoid of
$\mathbb{P}^{s-1}\cup\{[0]\}$ such that each element of $\mathbb{Y}$
is of the form $[\alpha]$ with $\alpha\in\{0,1\}^s$, 
then $\mathbb{Y}$ is parameterized by 
Laurent monomials.
\end{lemma}

\begin{proof}
The set $\mathbb{Y}$ can be written as
$\mathbb{Y}=\{[\alpha_1],\ldots,[\alpha_m]\}$, where
$\alpha_i=(\alpha_{i1},\ldots,\alpha_{is})$ and $\alpha_{ij}=0$ or
$\alpha_{ik}=1$ for all $i,k$. Consider variables $y_1,\ldots,y_s$ and
$z_1,\ldots,z_s$. For each $\alpha_{ik}$ define $h_{ik}=y_i^{q-1}$ if
$\alpha_{ik}=1$ and $h_{ik}=z_i^{q-1}/y_i^{q-1}$ if $\alpha_{ik}=0$.
Setting $h_i=(h_{i1},\ldots,h_{is})$ for $i=1,\ldots,m$ and 
$F_i=h_{1i}\cdots h_{mi}$ for $i=1,\ldots,s$, we get 
$$
h_1h_2\cdots h_m=(h_{11}\cdots h_{m1},\ldots,h_{1s}\cdots
h_{ms})=(F_1,\ldots,F_s).
$$
It is not hard to see that $\mathbb{Y}$ is parameterized by
$F_1,\ldots,F_s$.
\end{proof}

\begin{example} Let $K$ be the field $\mathbb{F}_3$ and 
let $\mathbb{Y}=\{[(1,1,0)],[0,1,1],[0,1,0],[1,1,1]\}$. With the
notation above, we get that $\mathbb{Y}$ is the projective set
parameterized by 
$$
F_1=(y_1z_2z_3)^2/(y_2y_3)^2,\, F_2=(y_1y_2y_3)^2,\,
F_3=(y_2z_1z_3)^2/(y_1y_3)^2.
$$
\end{example}

The next result gives a family of ideals where the converse of 
Corollary~\ref{puebla-cinvestav-finite-monomial} is true. 

\begin{proposition}\label{feb6-15-2} Let $K=\mathbb{F}_q$ be a finite field. 
If $\mathbb{Y}$ is a subset of $\mathbb{P}^{s-1}$ such that each
element of $\mathbb{Y}$ is of the form $[\alpha]$ with
$\alpha\in\{0,1\}^s$ and $I(\mathbb{Y})$ is a binomial ideal, then 
$\mathbb{Y}$ is a projective set parameterized by Laurent monomials.
\end{proposition}

\begin{proof} Since $\mathbb{Y}$ is finite, one has that
$\mathbb{Y}=\overline{\mathbb{Y}}=V(I(\mathbb{Y}))$, where
$\overline{\mathbb{Y}}$ is the Zariski closure and $V(I(\mathbb{Y}))$
is the zero set of $I(\mathbb{Y})$. Hence, as
$I(\mathbb{Y})$ is generated by binomials, we get that 
$\mathbb{Y}\cup\{[0]\}$  is a submonoid of
$\mathbb{P}^{s-1}\cup\{[0]\}$. Thus, by Lemma~\ref{feb3-15},
$\mathbb{Y}$ is parameterized by Laurent monomials. 
\end{proof}

This leads us to pose the following conjecture.

\begin{conjecture}\label{feb6-15-1}\rm Let $K=\mathbb{F}_q$ be a finite field and let
$\mathbb{Y}$ be a subset of $\mathbb{P}^{s-1}$. If $I(\mathbb{Y})$ is
a binomial ideal, then $\mathbb{Y}$ is a projective set parameterized by 
Laurent monomials. 
\end{conjecture}

In particular from Proposition~\ref{feb6-15-2} this conjecture is true
for $q=2$. 

\bibliographystyle{plain}

\end{document}